\documentclass[10pt]{amsart}

\usepackage{amsmath, amsfonts, amsthm, amssymb, graphicx, fullpage, enumerate, float, caption, subcaption, tikz, hyperref, scalerel, boldline, makecell, longtable, array}
\usetikzlibrary{graphs}  
   
\newtheorem{theorem}{Theorem}[section]        
\newtheorem{lemma}[theorem]{Lemma}
\newtheorem{corollary}[theorem]{Corollary}

\theoremstyle{remark}

\theoremstyle{definition}

\def\N{\mathbb{N}}
     
\def\Q{\mathbb{Q}}            
\def\R{\mathbb{R}}      
\def\Z{\mathbb{Z}}

\def\implies{\Longrightarrow}  
\def\iff{\Longleftrightarrow}

\begin{document}
\title{The sum-product problem for small sets} 
\author{Ginny Ray Clevenger,  \quad  Haley Havard, \quad  Patch Heard, \\  Andrew Lott, \quad   Alex Rice, \quad  Brittany Wilson}
 
\begin{abstract} For $A\subseteq \R$, let $A+A=\{a+b: a,b\in A\}$ and $AA=\{ab: a,b\in A\}$. For $k\in \N$, let $SP(k)$ denote the minimum value of $\max\{|A+A|, |AA|\}$ over all $A\subseteq \N$ with $|A|=k$. Here we establish $SP(k)=3k-3$ for $2\leq k \leq 7$, the $k=7$ case achieved for example by $\{1,2,3,4,6,8,12\}$, while $SP(k)=3k-2$ for $k=8,9$, the $k=9$ case achieved for example by $\{1,2,3,4,6,8,9,12,16\}$. For  $4\leq k \leq 7$, we provide two proofs using different applications of Freiman's $3k-4$ theorem; one of the proofs includes extensive case analysis on the product sets of $k$-element subsets of $(2k-3)$-term arithmetic progressions. For $k=8,9$, we apply Freiman's $3k-3$ theorem for product sets, and investigate the sumset of the union of two geometric progressions with the same common ratio $r>1$, with separate treatments of the overlapping cases $r\neq 2$ and $r\geq 2$.   
 
\end{abstract}
 
\address{Department of Mathematics, Millsaps College, Jackson, MS 39210}
\email{clevevr@millsaps.edu} 
\email{havarhe@millsaps.edu} 
\email{heardkp@millsaps.edu}
\email{andrew.lott@uga.edu}   
\email{riceaj@millsaps.edu} 
\email{wilsobn@millsaps.edu}

\maketitle  
\setlength{\parskip}{5pt}   
 
\section{Introduction}

For $A\subseteq \R$, we define the \textit{sumset} $A+A=\{a+b: a,b\in A\}$ and the \textit{product set} $AA=\{ab: a,b\in A\}$. It is a standard fact, and a pleasant exercise to verify, that if $A\subseteq \R$ and $|A|=k$, then \begin{equation}\label{sp} 2k-1 \leq |A+A| \leq \frac{k^2+k}{2},\end{equation} with equality holding on the left-hand side if and only if $A$ is an \textit{arithmetic progression}, a set of the form $\{x,x+d,\dots,x+(k-1)d\}$ with $x,d\in \R$ and $d>0$. Here and throughout the paper we use $|X|$ to denote the number of elements of a finite set $X$. The right-hand side of \eqref{sp} is precisely the number of pairs $(a,b)\in A\times A$ with $a\leq b$, so equality holds on the right-hand side if and only if $A$ has no repeated sums other than the ones guaranteed by commutativity. Such a set is known as a \textit{Sidon set}. 

By viewing multiplication as addition of exponents, we see that the same inequalities \eqref{sp} hold for $|AA|$, provided $A\subseteq (0,\infty)$. This time, equality holds on the left-hand side if and only if $A$ is a \textit{geometric progression}, a set of the form $\{x,rx,\dots,r^{k-1}x\}$ with $x>0$ and $r>1$. Since it is impossible to be both an arithmetic and geometric progression when $k\geq 3$ (essentially the arithmetic mean-geometric mean inequality), at least one of $|A+A|$ and $|AA|$ must exceed the minimum value of $2k-1$. But by how much?

In this direction, for $k\in \N$, we define $$SP(k)=\min_{\substack{A\subseteq \N \\ |A|=k}}\left(\max\{|A+A|,|AA|\}\right).$$ Since the question's introduction by Erd\H{o}s and Szemer\'edi \cite{ES} in 1983, an extensive literature has developed on the asymptotic behavior of $SP(k)$ as $k\to \infty$, referred to as the \textit{sum-product problem}. Erd\H{o}s and Szemer\'edi themselves showed $SP(k)\geq ck^{1+\delta}$ for constants $c,\delta>0$, at the same time conjecturing $SP(k)=k^{2-o(1)}$. Roughly and asymptotically speaking, the conjecture says one cannot do much better than $A=\{1,2,\dots,k\}$, which has $|A+A|=2k-1$ and $|AA|=k^2/(\log k)^{\delta+o(1)}$, where $\delta =1-(1+\log\log 2)/\log 2 \approx .086$ \footnote{We use $o(1)$ to denote a function tending to $0$ as $k\to \infty$, and we use $\log$ to denote the natural logarithm.}. The task of estimating $|AA|$ in this case is known as the \textit{Erd\H{o}s multiplication table problem}, for which the interested reader can refer to \cite{erdosMT} and \cite{Ford2}. Over the ensuing four decades, incremental progress has been made toward the Erd\H{o}s-Szemer\'edi conjecture, the best results (so far) coming through connections with incidence geometry (see, in chronological order, \cite{Nath}, \cite{Elekes}, \cite{Ford}, \cite{Soly}, \cite{KS2}, \cite{KS1}, \cite{Shak}, and \cite{RudSS}). The current best lower bound is due to Rudnev and Stevens \cite{RudStev}, who showed $SP(k)\geq k^{\frac{4}{3}+\frac{2}{1167}-o(1)}$.

In contrast with previous literature on the sum-product problem, we eschew growth estimates for $SP(k)$ as $k\to \infty$, opting instead for the more modest and elementary goal of precisely determining $SP(k)$ for small values of $k$. Our main results are summarized as follows.  

\begin{theorem}\label{mainSP} We have the following exact values for $SP(k)$: $$SP(k)=\begin{cases} 3k-3, & 2\leq k\leq 7 \\ 3k-2, & k=8,9 \end{cases}. $$
\end{theorem}

We note that, in addition to the trivial fact $SP(1)=1$, the $k=2,3$ cases of Theorem \ref{mainSP} are immediate, as $|A+A|=|AA|=3$ whenever $|A|=2$, and $\max\{|A+A|,|AA|\}=6$ whenever $|A|=3$, the latter by the aforementioned arithmetic mean-geometric mean inequality. Further, since $SP(k)$ is defined as a minimum, we can establish the required upper bounds for Theorem \ref{mainSP} by exhibiting a single example for each value of $k$, as is done in Table \ref{UBtable} below. We encourage the reader to verify the entries in the third and fourth columns, and we do not claim these examples to be unique, even up to scaling.   
 
\begin{table}[H]

\centering

\caption{Examples showing $SP(k)\leq 3k-3$ for $4\leq k \leq 7$ and  $SP(k)\leq 3k-2$ for $k=8,9$.}
		\label {UBtable}


\begin{tabular}{||c||c||c||c||}

\hline

$k$ & $A$ & $|A+A|$ & $|AA|$\\

\hline\hline

4 & \{1, 2, 3, 4\} & 7 & 9 \\

\hline

5 & \{1, 2, 3, 4, 6\} & 10 & 12 \\

\hline

6 & \{1, 2, 3, 4, 6, 8\} & 13 & 15 \\

\hline

7 & \{1, 2, 3, 4, 6, 8, 12\} & 18 & 18 \\

\hline

8 & \{1, 2, 3, 4, 6, 8, 9, 12\} & 20 & 22 \\

\hline

9 & \{1, 2, 3, 4, 6, 8, 9, 12, 16\} & 25 & 25 \\

\hline

\end{tabular}

\end{table}

We dedicate the remainder of the paper to establishing the required lower bounds on $SP(k)$ for $4\leq k \leq 9$. Our general approach is as follows: if one of $|A+A|$ or $|AA|$ is \textit{very} small (close to its minimum value), then known structural characterizations should allow us to conclude that the other is large (close to its maximum value). This core idea is not new, as seen in the \textit{few sums, many products problem}, dating to work of Elekes and Ruzsa \cite{ER}, and the \textit{few products, many sums problem}, for which the interested reader should refer to \cite{MRSS}.

  In Section \ref{3k4LB}, we provide two separate proofs for $4\leq k \leq 7$, each using a different application of Freiman's characterization of $k$-element sets with sumset size at most $3k-4$. The first proof, Section \ref{47e}, is the more elementary, and includes case analysis on the product sets of $k$-element subsets of $(2k-3)$-term arithmetic progressions, which may be of independent interest. The second proof, Section \ref{FIS}, is less labor intensive, but requires a few more mature tools, as well as the fact that a geometric progression of positive integers is a Sidon set. We provide a proof of this latter fact using the rational root theorem, the application of which inspired much of our work in Section \ref{3k3LB}. In that section, we tackle the trickier cases $k=8,9$ using Frieman's characterization of $k$-element sets with sumset size at most $3k-3$, and analysis of the sumsets of unions of two geometric progressions with the same common ratio $r>1$, with separate treatments for the overlapping cases $r\neq 2$ and $r\geq 2$. Most of these latter results are generalized to include geometric progressions with negative elements, and some may also be of independent interest. 

\section{Two proofs of $SP(k)\geq 3k-3$ for $4\leq k\leq 7$} \label{3k4LB} 

Establishing the inequality $SP(k)\geq 3k-3$ is equivalent to ruling out the existence of $A\subseteq \N$ with $|A|=k$ and $|A+A|,|AA|\leq 3k-4$. Fortunately, the right-hand side of the latter inequality provides access to the following precise characterization of Freiman \cite{Frei59}.

\begin{theorem}[Freiman's $3k-4$ theorem] \label{fo} If $A\subseteq \Z$ with $|A|=k$ and $|A+A|=2k-1+b\leq 3k-4$, then $A$ is contained in an arithmetic progression of length $k+b$.
\end{theorem} 

\noindent While the original paper \cite{Frei59} is in Russian, an English translation of Freiman's proof of Theorem \ref{fo}, which is an induction on $k$ requiring only clever counting, basic modular arithmetic, and patient case analysis, can be found on pages 12-13 of \cite{Frei73}. The argument does rely on the fact that the sets in question lie in $\Z$, which we discuss further in Section \ref{FIS}. 

\subsection{Theorem \ref{fo} applied on the sum side} \label{47e}   One approach to our goal for $4\leq k \leq 7$ is to assume $A\subseteq \N$ with $|A|=k$ and $|A+A|\leq 3k-4$, then try to show that $|AA|$ is necessarily at least $3k-3$. To this end, we apply Theorem \ref{fo} as written, with $b=k-3$, to conclude that $A$ is contained in a $(2k-3)$-term arithmetic progression. In other words, $A=\{x+ad: a\in \tilde{A}\}$ for some $x,d>0$ and some $\tilde{A}\subseteq \{0,1,\dots,2k-4\}$ with $|\tilde{A}|=k$. Changing $x$ if needed, we can assume $0\in \tilde{A}$ without loss of generality. In particular, $$AA=\{x^2+(a+b)dx+abd^2: a,b\in \tilde{A}\},$$ which we can think of as the image of $B=\{(a+b,ab): a,b\in \tilde{A}\}$ under the map $f(m,n)=x^2+mdx+nd^2$. While $f$ need not be injective on $B$, it does have the property that $f(m',n')>f(m,n)$ if $m'\geq m$, $n'\geq n$, and $\max\{m'-m,n'-n\}>0$, independent of $x$ and $d$. Therefore, it suffices to find a chain of elements $(m_1,n_1),\dots,(m_{3k-3},n_{3k-3})\in B$ satisfying $m_{i+1}\geq m_i$, $n_{i+1}\geq n_i$, and $\max\{m_{i+1}-m_i,n_{i+1}-n_i\}>0$ for all $1\leq i \leq 3k-4$, as such a chain corresponds to a strictly increasing sequence of $3k-3$ elements of $AA$.

\noindent To complete the proof, such a chain must be found for each $4\leq k \leq 7$ and each $\tilde{A}\subseteq \{0,1,\dots,2k-4\}$ with $0\in \tilde{A}$ and $|\tilde{A}|=k$. The number of options for $\tilde{A}$ is ${2k-4 \choose k-1}$, which is $4,15,56,210$ for $k=4,5,6,7$, respectively. We display one $k=6$ example in Figure \ref{APPfigure} below, and we include diagrams of all $285$ cases here: \url{https://github.com/andrewlott99/SP-k-}.

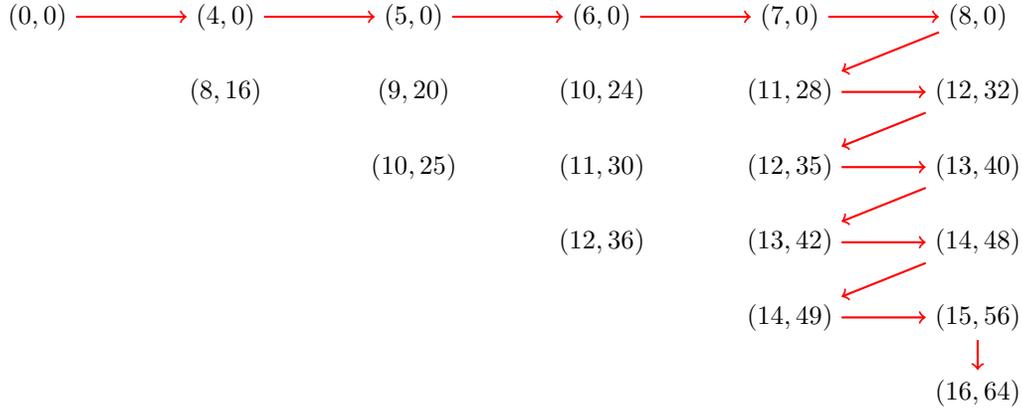
\begin{figure} [H] 
\begin{tikzpicture}

  \node (a) at (-7.5,10) {$(0,0)$};

  \node (b) at (-5,10) {$(4,0)$};

  \node (c) at (-2.5,10) {$(5,0)$};

  \node (d) at (0,10) {$(6,0)$};

  \node (e) at (2.5,10) {$(7,0)$};

  \node (f) at (5,10) {$(8,0)$};

  \node (g) at (-5,9) {$(8,16)$};

  \node (h) at (-2.5,9) {$(9,20)$};

  \node (i) at (0,9) {$(10,24)$};

  \node (j) at (2.5,9) {$(11,28)$};

  \node (k) at (5,9) {$(12,32)$};

  \node (l) at (-2.5,8) {$(10,25)$};

  \node (m) at (0,8) {$(11,30)$};

  \node (n) at (2.5,8) {$(12,35)$};

  \node (o) at (5,8) {$(13,40)$};

  \node (p) at (0,7) {$(12,36)$};

  \node (q) at (2.5,7) {$(13,42)$};

  \node (r) at (5,7) {$(14,48)$};

  \node (s) at (2.5,6) {$(14,49)$};

  \node (t) at (5,6) {$(15,56)$};

  \node (u) at (5,5) {$(16,64)$};

  \graph [edges={red, thick}] { (a) -> (b) -> (c) -> (d) -> (e) -> (f) -> (j) -> (k) -> (n) -> (o) -> (q) -> (r) -> (s) -> (t) -> (u)};

\end{tikzpicture}
\caption{The red path corresponds to a strictly increasing sequence of $15$ elements of $AA$ when $k=6$ and $\tilde{A}=\{0,4,5,6,7,8\}$. This is one of ${8 \choose 5}=56$ cases for $k=6$. In this and most other cases, longer qualifying paths exist, but we do not prioritize optimality in our analysis, only reaching $3k-3$.}
\label{APPfigure}
\end{figure}


\subsection{Theorem \ref{fo} applied on the product side}  \label{FIS} Another approach to our desired lower bound on $SP(k)$ is to assume $A\subseteq \N$ with $|A|=k$ and $|AA|\leq 3k-4$, then try to argue $|A+A|\geq 3k-3$. This begs the question as to whether an analog of Theorem \ref{fo} holds under the assumption of a small product set. In the introduction, we were able to immediately transfer the bounds \eqref{sp} from sumsets to products sets by viewing multiplication as addition of exponents. Things are not quite so simple here, because Freiman's proof of Theorem \ref{fo} is specific to $\Z$. Fortunately, this issue is well-trodden, and it is remarked repeatedly in the literature that Theorem \ref{fo} holds with $\Z$ replaced by any torsion-free abelian group (see the introductions of \cite{HLS} and \cite{BPS}, for example). For completeness, we include a proof of our desired special case, the meat of which can be found in Lemma 5.25 in \cite{TaoVu}.

\begin{corollary}\label{3k4tors} If $A\subseteq (0,\infty)$ with $|A|=k$ and $|AA|=2k-1+b\leq 3k-4$, then $A$ is contained in a geometric progression of length $k+b$.
\end{corollary}

\begin{proof}[Proof] Suppose $A\subseteq (0,\infty)$ with $|A|=k$ and $|AA|=2k-1+b\leq 3k-4$, and let $G$ be the multiplicative subgroup of $(0,\infty)$ generated by $A$. Since $G$ is torsion-free ($x\neq 1 \implies x^n\neq 1$ for all $n\in \N$), we know by the classification of finitely generated abelian groups that $G$ is isomorphic to $(\Z^j,+)$ for some $j\leq k$, say via $\psi:G\to \Z^j$. In particular, $|AA|=|\psi(A)+\psi(A)|$, and geometric progressions in $G$ correspond to arithmetic progressions in $\Z^j$.

\noindent After possibly translating,  we can assume  $\psi(A)\subseteq B=[0,L]^j\cap \Z^j$ for some $L\in \N$. Then, letting $M=2L$, the map $\varphi(x_1,\dots,x_j)=x_1+x_2M+\cdots+x_jM^{j-1}$ is a bijection satisfying $$a+b=c+d\iff \varphi(a)+\varphi(b)=\varphi(c)+\varphi(d)$$ for all $a,b,c,d\in B$. Such a map is known as a \textit{Freiman $2$-isomorphism}, which preserves both sumset sizes and arithmetic progressions, as the latter are defined by a system of equations of the form $x+y=z+z$. Finally, Theorem \ref{fo} yields that $\varphi(\psi(A))$ is contained in an arithmetic progression $P\subseteq \varphi(B) \subseteq \Z$ of length at most $k+b$, and hence $A$ is contained in the geometric progression $\psi^{-1}(\varphi^{-1}(P))$.
\end{proof}

\noindent Armed with Corollary \ref{3k4tors}, we immediately get a best-case lower bound on $|A+A|$ from the fact that a geometric progression of positive integers is a Sidon set. This is something of a folklore fact that we do not claim as original, but we include a proof at a level of generality that we do not believe has previously appeared, first recalling the rational root theorem. 

\begin{lemma}[Rational root theorem] Suppose $p(x)=a_kx^k+\cdots+a_1x+a_0\in \Z[x]$, $a_k\neq 0$. If $p(m/n)=0$ with $m,n\in \Z$, $\gcd(m,n)=1$, then $m\mid a_0$ and $n\mid a_k$. 
\end{lemma}

\noindent In the following lemma, and the remainder of the paper, we extend the definition of geometric progression to $\R$, allowing negative starting points and common ratios, excluding $0$ for both.
 
\begin{lemma} \label{sg} If $A\subseteq \R$ is contained in a geometric progression with common ratio $r\in \Q$, $r\neq -2$, then $A$ is a Sidon set. In particular, if $|A|=k$, then $|A+A|=(k^2+k)/2$.
\end{lemma}

\begin{proof} Suppose $A\subseteq \{r^nx\}_{n\in \Z}$ with $r,x\neq 0$, $r\in \Q$. If $r=-1$, then $|A|\leq 2$ and the result is trivial. Otherwise, we assume $|r|>1$, and a nontrivial repeated sum of elements of $A$ corresponds to an equation of the form $r^ax+r^bx=r^cx+r^dx$ with $a\geq b$, $c\geq d$, $b>d$. This rearranges to $r^{b-d}(r^{a-b}+1)-r^{c-d}-1=0$. If $c>d$, the rational root theorem forbids all $r\in \Q$ with $|r|>1$, while if $c=d$, it only allows $r=\pm 2$. However, if $r=2$, the left side of the equation is at least $2(2)-2=2$, so this case is ruled out as well. In other words, if $r\neq -2$, no such nontrivial repeated sum exists. \end{proof}

\noindent As for cases not covered in Lemma \ref{sg}, the conclusion fails for $r=-2$, as seen by the solution $4-2=1+1$, holds for irrational $r$ with $|r|>2$ because of the rapid growth of the progression, and fails for certain irrational $r$ with $1<|r|<2$. As an example of the latter, letting $\varphi=(1+\sqrt{5})/2$, we see that $\{1,\varphi^2,\varphi^3\}$ is a three-term arithmetic progression, in other words $1+\varphi^3=\varphi^2+\varphi^2$. We now combine Corollary \ref{3k4tors} and Lemma \ref{sg} to conclude the section.

\begin{corollary} Suppose $A\subseteq \N$ with $|A|=k$. If $|AA|\leq 3k-4$, then $|A+A|=(k^2+k)/2$. In particular, $SP(k)\geq 3k-3$ for all $k\in \N$. 
\end{corollary}

\section{Proof of $SP(k)\geq 3k-2$ for $k\geq 8$} \label{3k3LB} For $k\geq 8$, establishing the inequality $SP(k)\geq 3k-2$ is equivalent to ruling out the existence of $A\subseteq \N$ with $|A|=k$ and $|A+A|,|AA|\leq 3k-3$. Theorem \ref{fo} is no longer sufficient, but Freiman established another classification that perfectly fits the bill. The following is a less precise version of Theorem 1.11 in \cite{Frei73}. 

\begin{theorem}[Freiman's $3k-3$ theorem] \label{3k3o} If $A\subseteq \Z$ with $|A|=k>6$ and $|A+A|\leq 3k-3$, then either $A$ is contained in an arithmetic progression of length $2k+1$, or $A$ is a union of two arithmetic progressions of the same step size. 
\end{theorem} 

\noindent Completely analogous to the deduction of Corollary \ref{3k4tors} from Theorem \ref{fo}, we can appeal to Freiman isomorphism and apply Theorem \ref{3k3o} on the product side. 

\begin{corollary} \label{3k3p} If $A\subseteq (0,\infty)$ with $|A|=k>6$ and $|AA|\leq 3k-3$,  then either \begin{enumerate}[(i)] \item $A$ is contained in a geometric progression of length $2k+1$, or  \item $A$ is a union of two geometric progressions with the same common ratio. \end{enumerate}
\end{corollary}

We know from Lemma \ref{sg} that if $A$ falls into case (i) of Corollary \ref{3k3p}, then $A$ is a Sidon set and $|A+A|=(k^2+k)/2$. Therefore, we focus the remainder of the section on case (ii), and the results to come are summarized as follows. 

\begin{theorem} \label{gpsum} If $A\subseteq (0,\infty)$ with $|A|=k$ is a union of two geometric progressions with the same common ratio $r>1$, then $$|A+A|\geq \begin{cases} (k^2-3k+8)/2 & k \text{ odd}, \ r\in \Q, \ r\neq 2 \\ (k^2-3k+10)/2 & k \text{ even}, \ r\in \Q, \ r\neq 2 \\ \left\lceil ((k+1)^2+3)/4\right\rceil & r\geq 2 \\ 22 & k=8, \ r\in \Q, \ r\geq 2\end{cases}.  $$
\end{theorem}

\noindent Note that, for our purposes, the condition $r\in \Q$ is not restrictive, as the common ratio in a geometric progression of integers is guaranteed to be rational. We make the distinction here because one of our proofs does not rely on the rationality of $r$. Crucially, for every $k\geq 8$ and every $r\in \Q$, $r>1$, at least one case in the conclusion of Theorem \ref{gpsum} guarantees $|A+A|\geq 3k-2$, yielding our desired result. 

\begin{corollary}\label{83k2} For all $k\geq 8$, we have $SP(k)\geq 3k-2$.
\end{corollary}

\noindent The $r\neq 2$ cases of Theorem \ref{gpsum} are covered in Section \ref{not2}, while the $r\geq 2$ cases are covered in Section \ref{geq2}.

\subsection{The $r\neq 2$ case} \label{not2}  In the proofs that follow, we make frequent use of the fact that if $A\subseteq \Z$ is finite with $A=B\cup C$, then $A+A=(B+B)\cup(C+C)\cup(B+C)$, and hence
\begin{equation} \label{iep} |A+A|\geq |B+B|+|C+C|+|B+C|-|(B+B)\cap(C+C)|-|(B+B)\cap(B+C)|-|(C+C)\cap(B+C)|\end{equation} by the inclusion-exclusion principle. In the case of interest, $B$ and $C$ are geometric progressions of the same common ratio $r>1$. The pairwise intersections in \eqref{iep} can certainly be nonempty, and elements in $B+C$ can certainly have multiple representations, pulling $|B+C|$ away from its maximum. However, if $r\in \Q$ and $r\neq 2$, then such coincidences are limited to a single ``geometric family" in each case. This fact, shown through repeated applications of the rational root theorem, is captured in the following lemma, which also allows for negative starting points and common ratios.






\begin{lemma} \label{gf} For fixed $r\in \Q$ with $1<|r|\neq 2$, and $z\in \R \setminus \{r^n: n\in \Z\}$ with $z\neq 0$, there is at most one solution $(a,b,c)$ to each of the  following equations under the given restrictions, with the listed exceptions: 
\begin{enumerate}[(i)] \item $r^a+1=(r^b+r^c)z$, \ $a,b,c\in \Z$, $a\geq 0$, $b\geq c$,  \ \text{except when} $r=-3$ \text{and} $z=-r^n$ \text{for some} $n\in \Z$, 
\item $r^a+1=r^b+r^cz$, \ $a,b,c\in \Z$, $a\geq 0$, \ \text{except when} $r=-3$ \text{and} $z=-5r^n/3$ \text{for some} $n\in \Z$, or 

\qquad \qquad \qquad \qquad \qquad \qquad \qquad \quad \  $r=-3/2$ and $z=14r^n/9$ for some $n\in \Z$,
\item $r^a-1=(r^b-r^c)z$, \ $a,b,c\in \Z$, $a> 0$. The conclusion also holds for (iii) when $r=2$, but not $r=-2$.
\end{enumerate}
\end{lemma}

\begin{proof} Suppose $r\in \Q$ with $|r|>1$, $|r|\neq 2$, $z\in \R \setminus \{r^n: n\in \Z\}$, and $z\neq 0$. 

\begin{enumerate}[(i)] \item Suppose there exists a solution of the form $r^a+1=(r^b+r^c)z$ with $a,b,c\in \Z$, $a\geq 0$, $b\geq c$. Replacing $r^cz$ with $z$ and $b-c$ with $b$ for convenience without loss of generality, we have $r^a+1=(r^b+1)z$ with $b\geq 0$. Since $z\neq 1$, $a$ and $b$ cannot both be $0$, so we assume without loss of generality that $a>0$, otherwise replacing $z$ with $1/z$. \\[5pt]
\noindent Now, suppose we have another solution $r^{a'}+1=(r^{b'}+r^{c'})z$ with $a',b',c'\in \Z$, $a'\geq 0$, $b'\geq c'$. Cross multiplying the two solutions and canceling $z$, this yields \begin{equation} \label{p1} r^{c'}(r^a+1)(r^{b'-c'}+1)=(r^b+1)(r^{a'}+1),\end{equation} noting that all exponents in \eqref{p1} are nonnegative with the possible exception of $c'$, while only $a$ is necessarily positive. \\[5pt]
\noindent \textbf{Case 1:} $c'\neq 0$. If $c'>0$, then the left side of $\eqref{p1}$ has constant term $0$ and leading term $2r^{a+b'}$ if $b'=c'$ and $r^{a+b'}$ otherwise. Meanwhile, the right hand side is either constant $4$, achieved only if $b=a'=0$, or it has leading and constant coefficient $1$ or $2$. In any case \eqref{p1} can be rearranged to an integral polynomial, set equal to $0$, with leading coefficient at most $2$ in absolute value, and constant coefficient the same as that of $(r^b+1)(r^{a'}+1)$. If $b=a'=0$, the rational root theorem allows for the possibility $r=\pm4$. However, substituting $r=\pm4$ into \eqref{p1} in this case yields  $((\pm 4^a)+1)((\pm 4^{b'-c'})+1)=\pm 4^{1-c'}$, which is impossible because, since $a>0$ and $b'-c'\geq 0$, the left side is an integer with an odd divisor greater than $1$. Otherwise, regardless of $b$ and $a'$, the only rational possibilities with $|r|>1$ are $r=\pm 2$, which are excluded by hypothesis. Similar reasoning applies to the case $c'<0$ by moving $r^{c'}$ to the other side of the equation, except that now it is the leading coefficient that could have absolute value as high as $4$, which does not provide any other possible rational $r$ with $|r|>1$. \\[5pt]
\noindent \textbf{Case 2:} $c'=0$. Here we have \begin{equation}\label{p12} (r^a+1)(r^{b'}+1)=(r^b+1)(r^{a'}+1), \end{equation} all exponents are nonnegative, and $a>0$. \\[5pt]
\noindent \textbf{Case 2a:} $b=a'=0$. In this case, \eqref{p12} rearranges to $r^{a+b'}+r^a+r^{b'}-3=0$. If $b'=0$, this collapses to $2(r^a-1)=0$, forcing $|r|=1$. If $b'>0$, the leading coefficient is $1$ and the constant term is $-3$, so the only possible rational roots with $|r|>1$ are $r=\pm 3$. When $r=3$, the left-hand side is clearly bigger than the right, but for $r=-3$, there is one solution with $a=b'=1$, in other words $(-3+1)(-3+1)=(1+1)(1+1)$. This accounts for the exception in item (i) in the lemma.    \\[5pt]
\noindent \textbf{Case 2b:} $b+a'>0$. If the two sides of \eqref{p12} are not identical (meaning equal as polynomials), then, after canceling common factors of $r$, we have an integral polynomial, set equal to $0$, with nonzero constant term at most $2$ in absolute value, ruling out rational $r$ with $1<|r|\neq 2$. The remaining case is the two sides of \eqref{p12} are identical, hence either $a=b$, which implies $z=1$ so is prohibited, or $a=a'$ and $b=b'$, in which case the two solutions to (i) are in fact the same. 

\

\item We make a similar substitution and suppose we have two solutions $r^a+1=r^b+z$ and $r^{a'}+1=r^{b'}+r^{c'}z$ with $a,b,a',b',c'\in \Z$, $a,a'\geq 0$. Note that $b,b'\neq 0$ since $z$ is not a power of $r$. Solving both equations for $z$ and setting them equal to each other yields \begin{equation}\label{p21} r^{c'}(r^a+1-r^b)=r^{a'}+1-r^{b'}. \end{equation} 

\noindent \textbf{Case 1:} $bb'<0$. We assume $b<0$ and $b'>0$, with the opposite case handled identically. Then, we have $r^{c'+b}(r^{a-b}+r^{-b}-1)=r^{a'}+1-r^{b'}$, with all exponents nonnegative except possibly $c'+b$. \\[5pt]
\noindent \textbf{Case 1a:} $c'+b\neq 0$. If $c'+b>0$, we have an integral polynomial, set equal to $0$, with constant term the same as that of $r^{a'}+1$, which is $1$ or $2$, ruling out rational $r$ with  $1<|r|\neq 2$. The case $c'+b<0$ is handled identically after dividing $r^{c'+b}$ to the other side. \\[5pt]
\noindent \textbf{Case 1b:} $c'+b=0$. Here we have $r^{a-b}+r^{-b}+r^{b'}-r^{a'}-2=0$, with all exponents positive except possibly $a'$. If $a'> 0$, then the constant term is $-2$, ruling out rational $r$ with $1<|r|\neq 2$. If $a'=0$, we have $r^{a-b}+r^{-b}+r^{b'}-3=0$, with all exponents positive. If $a=b+b'=0$, this collapses to $3(r^{b'}-1)=0$, forcing $|r|=1$. If there is a unique maximum amongst $a-b$, $-b$, and $b'$, then the leading coefficient is $1$ and the constant term is $-3$, so the only possible rational $r$ with $|r|>1$ are $r=\pm3$. If $r=3$, the left side is clearly bigger, but there is a solution with $r=-3$, namely $b'=-b=a=1$, which accounts for the first exception in item (ii). If instead two out of $a-b$, $-b$, and $b'$ are equal, while the third is smaller, we have an equation of the form $r^{j}(2r^{\ell}+1)=3$ with $j,\ell>0$. The left side is clearly bigger in absolute value for $r>1$ and $r=-3$, and all other rational $r$ with $|r|>1$ are ruled out except $r=-3/2$. However, there is a solution with $r=-3/2$, namely $j=\ell=1$, which corresponds to $a=0$, $b=-2$, and $b'=1$, and accounts for the second exception in item (ii). \\[5pt]
\noindent \textbf{Case 2:} $b,b'>0$. If $c'\neq 0$, we get an integral polynomial, set equal to $0$, with nonzero constant term at most $2$ in absolute value. If $c'=0$, we have $r^a+r^{b'}=r^b+r^{a'}$. Lemma \ref{sg} then implies either $a=b$, which yields $z=1$ so is prohibited, or $a=a'$ and $b=b'$, so the two solutions to (ii) are the same. \\[5pt]  
\noindent \textbf{Case 3:} $b,b'<0$. Here we have $r^{c'+b-b'}(r^{a-b}+r^{-b}-1)=r^{a'-b'}+r^{-b'}-1$, with all exponents positive except possibly $c'+b-b'$. If $c'+b-b'\neq 0$,  we get an integral polynomial, set equal to $0$, with constant term $1$, ruling out rational $r$ with $|r|>1$. If $c'+b-b'=0$, we have $r^{a-b}+r^{-b}=r^{a'-b'}+r^{-b'}$, which by Lemma \ref{sg} implies either $b=b'$ and $a=a'$, and hence $c'=0$, as desired, or $a-b=-b'$ and $-b=a'-b'$. In the latter case, we have $a=b-b'=-a'$, but since $a,a'\geq 0$ this implies $a=a'=0$, so $b=b'$ and $c'=b'-b=0$, completing the proof for (ii). 

\ 

\item In this item, we do not invoke $r\neq 2$. We again substitute and suppose  $r^a-1=(r^b-1)z$ and $r^{a'}-1=(r^{b'}-r^{c'})z$ with $a,b,a',b',c'\in \Z$, $a,a'>0$. Multiplying, and canceling $z$, gives $$(r^a-1)(r^{b'}-r^{c'})=(r^{a'}-1)(r^{b}-1).$$ Let $m=\max\{0,-b,-b',-c'\}$, and multiply both sides by $r^m$ to yield \begin{equation}\label{p3} (r^a-1)(r^{b'+m}-r^{c'+m})=(r^{a'}-1)(r^{b+m}-r^m),\end{equation} where now all exponents are nonnegative and $\min\{m,b+m,b'+m,c'+m\}=0$. If this minimum is uniquely attained, we get an integral polynomial, set equal to $0$, with constant term $1$, ruling out rational $r$ with $|r|>1$. Otherwise, it must be that exactly one of $b'+m,c'+m$ is $0$ and exactly one of $b+m,m$ is $0$. By symmetry, it suffices to consider the following two cases: \\[5pt]
\noindent \textbf{Case 1:} $c'=m=0$, $b,b'>0$. Returning to \eqref{p3}, we have \begin{equation} \label{p32}(r^a-1)(r^{b'}-1)=(r^{a'}-1)(r^{b}-1),\end{equation} with all exponents positive. If the two sides of \eqref{p32} are not identical, then either there is a unique minimum in $\{a,b,a',b'\}$, and the equation rearranges to an integral polynomial, set equal to $0$, with constant term $1$, or one of $a=b'$ or $b=a'$ are strictly less than the other pair of exponents. In the latter case, \eqref{p32} clearly does not hold for $r>1$, and further \eqref{p32} can be rearranged to an integral polynomial, set equal to $0$, with constant term $2$, so the only possibly rational $r$ with $|r|>1$ is $r=-2$, which is excluded by hypothesis. If the two sides of \eqref{p32} are identical, then either $a=b$, which implies $z=1$ so is prohibited, or $a=a'$ and $b=b'$, so the two solutions to (iii) are in fact the same. \\[5pt]
\noindent \textbf{Case 2:} $c'=b=-m$, $m,b'+m>0$. Returning to \eqref{p3}, we have \begin{equation} \label{p33} (r^a-1)(r^j-1)=(r^{a'}-1)(1-r^m), \end{equation} where $j=b'+m$, and all exponents are positive. This rearranges to an integral polynomial, set equal to $0$, with constant term $2$, so the only possible rational $r$ with $|r|>1$ are $r=\pm 2$. However, if $r>1$, then the left side of \eqref{p33} is positive while the right side is negative, so $r=2$ can be ruled out without invoking a hypothesis, which at long last completes the proof of the lemma.
 \end{enumerate} \end{proof} 
\noindent Parts (i) and (ii) of Lemma \ref{gf} fail when $r=2$, as seen by type (i) solutions $2^3+1=(2+1)3$ and $2+1=(2^{-1}+2^{-1})3$, and type (ii) solutions $2^2+1=2+3$ and $2+2=1+3$. Our application of Lemma \ref{gf} is quantitative, but we first state its consequences in a qualitative form that may be of independent interest.  

\begin{corollary}\label{intcor} Suppose $r\in \Q$ with $|r|>1$, $r\neq -2$, $x,y\neq 0$, $B=\{r^nx\}_{n\in \Z}$,   $C=\{r^ny\}_{n\in \Z}$, and $B\neq C$. 

\begin{enumerate}[(i)] \item For every $a\neq 0$, we have $R(a)=|\{(b,c)\in B\times C: b+c=a\}|\leq 2$. 
\item The set $\{a\in B+C: R(a)>1\}$ is either empty, $\{0\}$, or a geometric progression with common ratio $r$. 
\item If $r\notin \{2,-3/2,-3\}$, then $(B+B)\cap(C+C)$, $(B+B)\cap (B+C)$, and $(C+C)\cap (B+C)$ are each either empty or a geometric progression with common ratio $r$, with the elements in the latter two intersections uniquely represented as $b+c$ with $b\in B$, $c\in C$.
\end{enumerate}
\end{corollary}

\begin{proof} Each element in the three intersections of interest, and each repeated representation in $B+C$, yields an equation that uniquely scales by a power of $r$ to an equation treated in Lemma \ref{gf}, with $z=y/x$ or $z=x/y$ as appropriate. Since at most one such solution of each type exists, the result follows. The special case $a=0$ in items (i) and (ii) comes from the case $B=-C$. \end{proof} 

\noindent The following corollary captures our quantitative application of Lemma \ref{gf}. The $r\in \Q$, $r\neq 2$ cases in Theorem \ref{gpsum} then follow by taking $m=\lceil k/2 \rceil$ and $n=\lfloor k/2 \rfloor$, which is the worst-case scenario.  

\begin{corollary} \label{gp8} Suppose $x,y\neq 0$ and $r\in \Q$ with $|r|>1$, $r\notin \{2,-3/2,-2,-3\}$. If $A=B\cup C$ with $B=\{x,rx,\dots,r^{m-1}x\}$, $C=\{y,ry,\dots,r^{n-1}y\}$, $B\cap C=\emptyset$,  and $m\geq n>0$, then 
$$|A+A|\geq ((m+n)^2+m+n)/2-\min\{m-1-\alpha,n-1\}-2\min\{m-1,n\}-(n-1), $$ where $\alpha=0$ if $0\in B+C$ and $\alpha=1$ otherwise.
\end{corollary}

\begin{proof} Suppose $x,y\neq 0$ and $r\in \Q$ with $|r|>1$, $|r|\neq 2$. Suppose $A=B\cup C$, where $B=\{x,rx,\dots,r^{m-1}x\}$, $C=\{y,ry,\dots,r^{n-1}y\}$, $B\cap C=\emptyset$, $0\notin B+C$, and $m\geq n>0$. We first note that \begin{equation}\label{BBnCC} |B+B|=(m^2+m)/2, \quad |C+C|=(n^2+n)/2 \end{equation} by Lemma \ref{sg}. That same lemma assures that if $y=r^jx$ for some $j\in \Z$, then $A$ is contained in a single geometric progression and hence $|A+A|=((m+n)^2+m+n)/2$, so we assume this not to be the case for the remainder of the proof. 

\noindent By Lemma \ref{gf}, the elements of $(B+B)\cap(C+C)$, if there are any, correspond to a single family of solutions $(r^{a_0+j}+r^{b_0+j})x=(r^{c_0+j}+r^{d_0+j})y$, $0\leq b_0+j \leq a_0+j <m$, $0\leq d_0+j \leq c_0+j <n$, $\min\{a_0,b_0,c_0,d_0\}=0$, stemming from the at most one solution of type (i) with $z=y/x$. Since $\max\{a_0,c_0\}\geq 1$, we see that \begin{equation}\label{BBCC} |(B+B)\cap(C+C)|\leq \min\{m-1,n\}.\end{equation} An example with $m=n=4$ is shown in the figure below:

\begin{figure} [H] 
\begin{tikzpicture}[dot/.style={circle, fill, inner sep=1pt}]

   \draw (0,1) node[dot] (1) {} node[below] {1};
  \draw (1,1) node[dot] (3) {} node[below] {3};
  \draw (2,1) node[dot] (9) {} node[below] {9};
  \draw (3,1) node[dot] (27) {} node[below] {27};
 
  \foreach \x in {0,1,2,3} {
    \draw (\x,1) node[dot] {};
    \ifnum\x=0
      \draw[->, bend left=80, purple] (\x,1) to (\x+1,1);
    \fi
    \ifnum\x=1
      \draw[->, bend left=80, orange] (\x,1) to (\x+1,1);
    \fi
    \ifnum\x=2
      \draw[->, bend left=80, yellow] (\x,1) to (\x+1,1);
    \fi
  }

  \draw (.5,-.5) node[dot] (2) {} node[below] {2};
  \draw (1.5,-.5) node[dot] (6) {} node[below] {6};
  \draw (2.5,-.5) node[dot] (18) {} node[below] {18};
  \draw (3.5,-.5) node[dot] (54) {} node[below] {54};

  \draw[->, purple, out=130, in=50, distance=1cm] (2) to (2);
  \draw[->, orange, out=130, in=50, distance=1cm] (6) to (6);
  \draw[->, yellow, out=130, in=50, distance=1cm] (18) to (18);

\end{tikzpicture}
\caption{An illustration of $(B+B)\cap(C+C)=\{4,12,36\}$, where $B=\{1,3,9,27\}$ and $C=\{2,6,18,54\}$. Each color shows a member of the single geometric family of solutions, stemming from the solution $3+1=2+2$.}
		\label {test1}	 
\end{figure}
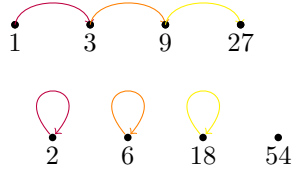
\noindent Similarly, potential elements of $(B+B)\cap(B+C)$ correspond to a family $(r^{a_0+j}+r^{b_0+j})x=r^{c_0+j}x+r^{d_0+j}y$ with  $0\leq b_0+j \leq a_0+j <m$, $0\leq c_0+j <m$, $0\leq d_0+j <n$, $\min\{a_0,b_0,c_0,d_0\}=0$, stemming from the at most one solution of type (ii) with $z=y/x$. Once again $\max\{a_0,c_0\}\geq 1$, and the number of solutions is at most $\min\{m-1,n\}$. The same reasoning applies to elements of $(C+C)\cap(B+C)$, now with $z=x/y$ and the terms $r^{a_0}y,r^{c_0}y$ with $\max\{a_0,c_0\}\geq 1$, so the number of solutions is at most $n-1$. To summarize, \begin{equation}\label{BBBC}|(B+B)\cap(B+C)|\leq \min\{m-1,n\}, \quad |(C+C)\cap(B+C)| \leq n-1. \end{equation}
See the figure below for another example when $m=n=4$.
\begin{figure} [H]
\begin{tikzpicture}[dot/.style={circle, fill, inner sep=1pt}]

   \draw (0,1) node[dot] (8) {} node[above] {8};
  \draw (1,1) node[dot] (12) {} node[above] {12};
  \draw (2,1) node[dot] (18) {} node[above] {18};
  \draw (3,1) node[dot] (27) {} node[above] {27};
 
  \foreach \x in {0,1,2,3} {
    \draw (\x,1) node[dot] {};
    \ifnum\x=0
      \draw[->, purple] (\x,1) to (\x+1.5,-.5);
    \fi
    \ifnum\x=1
      \draw[->, orange] (\x,1) to (\x+1.5,-.5);
    \fi
    \ifnum\x=2
      \draw[->, yellow] (\x,1) to (\x+1.5,-.5);
    \fi
  }

  \draw (1.5,-.5) node[dot] (16) {} node[below] {16};
  \draw (2.5,-.5) node[dot] (24) {} node[below] {24};
  \draw (3.5,-.5) node[dot] (36) {} node[below] {36};
  \draw (4.5,-.5) node[dot] (54) {} node[below] {54};

  \draw[->, purple, out=-130, in=-50, distance=1cm] (12) to (12);
  \draw[->, orange, out=-130, in=-50, distance=1cm] (18) to (18);
  \draw[->, yellow, out=-130, in=-50, distance=1cm] (27) to (27);

\end{tikzpicture}
\caption{An illustration of $(B+B)\cap(B+C)=\{24,36,54\}$, where $B=\{8,12,18,27\}$ and $C=\{16,24,36,54\}$. Each color shows a member of the single geometric family of solutions, stemming from the solution $12+12=8+16$.}
		\label {geofamily2} 	 
\end{figure}
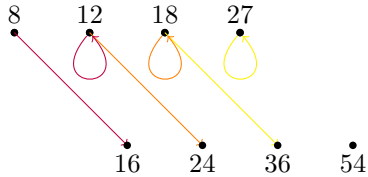
\noindent Finally, of the $mn$ pairs $(b,c)\in B+C$, the only loss of distinct sums comes from a geometric family $$r^{a_0+j}x+r^{d_0+j}y=r^{b_0+j}x+r^{c_0+j}y=r^jw,$$ where $w=r^{a_0}x+r^{d_0}y$, $0< b_0+j \leq a_0+j <m$, $0\leq d_0+j < c_0+j <n$, $\min\{a_0,b_0,c_0,d_0\}=0$, stemming from the at most one solution of type (iii) with $z=y/x$. If $w\neq 0$, then since $a_0,c_0\geq 1$ and $\max\{a_0,c_0\}\geq 2$, we have $p\leq \min\{m-2,n-1\}$ pairs of pairs, each collapsing to a different common sum, so $|B+C|$ loses $p$ from its maximum value. If instead $w=0$, we have $p\leq n$ pairs collapsing to the single sum $0$, so $|B+B|$ loses $p-1$ from its maximum value.  In other words, \begin{equation}\label{BPC} |B+C|\geq mn-\min\{m-1-\alpha,n-1\}, \end{equation} where $\alpha$ is as stated in the corollary. We illustrate the $w\neq 0$ case below with one final $m=n=4$ example.
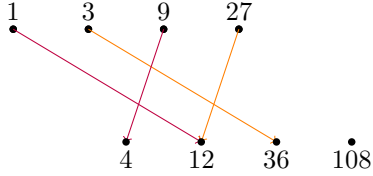
\begin{figure} [H]
\begin{tikzpicture}[dot/.style={circle, fill, inner sep=1pt}]

   \draw (0,1) node[dot] (1) {} node[above] {1};
  \draw (1,1) node[dot] (3) {} node[above] {3};
  \draw (2,1) node[dot] (9) {} node[above] {9};
  \draw (3,1) node[dot] (27) {} node[above] {27};
 
  \foreach \x in {0,1,2,3} {
    \draw (\x,1) node[dot] {};
    \ifnum\x=0
      \draw[->, purple] (\x,1) to (\x+2.5,-.5);
    \fi
    \ifnum\x=1
      \draw[->, orange] (\x,1) to (\x+2.5,-.5);
    \fi
     \ifnum\x=2
      \draw[->, purple] (\x,1) to (\x-.5,-.5);
    \fi
    \ifnum\x=3
      \draw[->, orange] (\x,1) to (\x-.5,-.5);
    \fi
  }

  \draw (1.5,-.5) node[dot] (4) {} node[below] {4};
  \draw (2.5,-.5) node[dot] (12) {} node[below] {12};
  \draw (3.5,-.5) node[dot] (36) {} node[below] {36};
  \draw (4.5,-.5) node[dot] (108) {} node[below] {108};

\end{tikzpicture}
\caption{An illustration showing that if $B=\{1,3,9,27\}$ and $C=\{4,12,36,108\}$, then the only elements of $B+C$ with two distinct representations are $13$ and $39$, stemming from the solution $9-1=12-4$. In particular, $|B+C|=(4)(4)-2=14$.}
		\label {geofamily3}	 
\end{figure} 
\noindent Substituting \eqref{BBnCC}, \eqref{BBCC}, \eqref{BBBC}, and \eqref{BPC} into \eqref{iep} yields the claimed lower bound on $|A+A|$. \end{proof}

\noindent As with Lemma \ref{gf}, the conclusion of Corollary \ref{gp8} fails when $r=2$, as seen by $B=\{1,2,4,8,16\}$, $C=\{3,6,12\}$.
  
\subsection{The $r\geq 2$ case} \label{geq2} The following lemma makes precise the fact that a union of two geometric progressions with the same common ratio $r\geq 2$ determines many distinct sums based only on the rapid growth in the terms. This is our one result in Section \ref{3k3LB} that does not rely on the common ratio being a rational number.

\begin{lemma} \label{2gplem1}
If $A\subseteq (0,\infty)$ with $|A|=k$ is a union of two geometric progressions with the same common ratio $r\geq 2$, then
$$|A+A|\geq \left\lceil \frac{(k+1)^2+3}{4}\right\rceil.$$
\end{lemma}
\begin{proof}
Suppose $A\subseteq (0,\infty)$ with $|A|=k$ is a union of two geometric progressions with the same common ratio $r\geq 2$. Write the elements of $A$ in increasing order $a_1<a_2<\cdots<a_{k}.$ Given any three consecutive elements $a_{m-2}, a_{m-1}, a_m,$ at least two must come from the same geometric progression, hence $$a_m\geq ra_{m-2}\geq 2a_{m-2}$$ for all $2\leq m \leq k$. Therefore, if $k$ is even, then
$$\{a_{k}+a_1, a_{k}+a_2, \dots, 2a_{k}\}, \{a_{k-2}+a_1, a_{k-2}+a_2, \dots, 2a_{k-2}\}, \dots, \{a_2+a_1,2a_2\}, \{2a_1\} $$
are pairwise disjoint sets of distinct sums, and if $k$ is odd, then
$$ \{a_{k}+a_1, a_{k}+a_2, \dots, 2a_{k}\},\{a_{k-2}+a_1, a_{k-2}+a_2, \dots, 2a_{k-2}\}, \dots, \{a_3+a_1,a_3+a_2,2a_3\},\{a_1+a_2,2a_1\}$$
are pairwise disjoint sets of distinct sums. The total number of distinct sums listed is $$1+2\sum_{j=1}^{k/2}j=1+\frac{k(k+2)}{4}=\frac{(k+1)^3+3}{4}$$ when $k$ is even and $$1+\frac{k+1}{2}+2\sum_{j=1}^{(k-1)/2}j=1+\frac{k+1}{2}+\frac{(k-1)(k+1)}{4}=\frac{(k+1)^2}{4}+1$$ when $k$ is odd, yielding the claimed lower bound in both cases.
\end{proof}

\noindent The right-hand side of Lemma \ref{2gplem1} exceeds $3k-2$ for $k\geq 9$, but is only $21$ when $k=8$. Fortunately, the additional assumption that the common ratio is a rational number is enough to get us across the goalline. 

\begin{lemma}

If $A\subseteq (0,\infty)$ with $|A|=8$ is a union of two geometric progressions with the same rational common ratio $r\geq 2$, then $|A+A|\geq 22$.

\end{lemma}
\begin{proof}
Suppose $A\subseteq (0,\infty)$ with $|A|=8$ and $A=B\cup C$, where $B$ and $C$ are geometric progressions of the same common ratio $r\in \Q$, $r\geq 2$. Write the elements of $A$ in increasing order $a_1<a_2<\cdots<a_{8}.$
We are guaranteed the $21$ distinct sums listed in the proof of Lemma \ref{2gplem1}, the set of which we call $S$, so we must identify one element of $(A+A)\setminus S$.

\noindent
If there is some $j\in\{2,\dots, 7\}$ such that $a_j, a_{j-1}$ are both in $B$ or both in $C$, then $a_{j}+a_1\notin S$ if  $j$ is even and $2a_j\notin S$ if $j$ is odd. Otherwise, we assume without loss of generality that  $a_2,a_4,a_6,a_8\in B$, while $a_3,a_5,a_7\in C$. If $a_3+a_1\neq2a_2$, then $a_3+a_1\notin S$, as needed.
Otherwise, $a_3+a_1=2a_2$, and we consider the cases $a_1\in B$ and $a_1\in C$.

\noindent
If $a_1\in B$, then $$(a_1,a_2,a_3,a_4,a_5,a_6,a_7,a_8)=(x,rx,y,r^2x,ry,r^3x,r^2y,r^4x)$$ with $x,y>0$, 
and $a_3+a_1=2a_2$ gives $y=(2r-1)x.$ Substituting $y=(2r-1)x$ and comparing $a_7+a_1$ to each element of $S$ that it could potentially equal, namely $a_6+a_2,a_6+a_3,\dots, 2a_6, $ we see that any match yields a polynomial equation $p(r)=0$ with no rational solution $r\geq 2$. Thus, $a_7+a_1\notin S$, as needed.

\noindent
If $a_1\in C$, then $$(a_1,a_2,a_3,a_4,a_5,a_6,a_7,a_8)=(y,x,ry,rx,r^2y,r^2x,r^3y,r^3x)$$ with $x,y>0$,
and $a_3+a_1=2a_2$ gives $x=(y+ry)/2$. The only element of $S$ that could possibly equal $a_3+a_2$ is $a_4+a_1$. However, assuming  $a_3+a_2= a_4+a_1$ and substituting $x=(y+ry)\slash 2$ yields $(r-1)^2=0$, contradicting $r\geq 2$. \end{proof}

\section{Concluding remarks} All cases in the conclusion of Theorem \ref{gpsum} have now been established, which in turn completes the proofs of Corollary \ref{83k2} and our main result Theorem \ref{mainSP}. From here, a natural question is whether our methods can be modified or supplemented to determine $SP(10)$, or beyond. At the moment, this feels out of reach. Based on both hand and computer searches, our best guess is that $SP(10)=30$, as achieved by $A=\{1,2,3,4,6,8,9,12,16,18\}$, which has $|A+A|=30$ and $|AA|=29$. However, there is not a characterization of the precision of Theorems \ref{fo} and \ref{3k3o} that can be applied to a set $A\subseteq \N$ with $|A|=10$ and $|A+A|$ or $|AA|$ equal to $29$. Some structural results beyond $3k-3$ have been established, for example Jin \cite{Jin07} up to $3k-3+\epsilon k$ for some small $\epsilon>0$, and Eberhard, Green, and Manners \cite{EGM} up to $(4-\epsilon)k$ for every $\epsilon>0$. However, these results require $k$ to be sufficiently large, so they are, a priori, not useful for our purposes. Development of new tools, or novel application of existing tools, to determine $SP(k)$ for some values of $k>9$ could be an interesting endeavor for future work.

\noindent \textbf{Acknowledgements:} This research was initiated during the Summer 2023 Kinnaird Institute Research Experience at Millsaps College. All authors were supported during the summer by the Kinnaird Endowment, gifted to the Millsaps College Department of Mathematics. At the time of submission, all authors except Alex Rice and Andrew Lott were Millsaps College undergraduate students.


\begin{thebibliography}{10} 

\bibitem{BPS} {\sc K. B\"or\"oczky, P. P\'alfy, O. Serra}, {\em On the cardinality of sumsets in torsion-free groups} Bulletin of the London Math. Soc. 44 (2012), no. 5, 1034-1041.
\bibitem{EGM} {\sc S. Eberhard, B. Green, F. Manners}, {\em Sets of integers with no large sum-free subset}, Annals of Mathematics 180 (2014), no. 2, 621-652.
\bibitem{Elekes} {\sc G. Elekes}, {\em On the number of sums and products}, Acta Arithmetica 81 (1997), no. 4, 365-367.
\bibitem{ER} {\sc G. Elekes, I. Ruzsa}, {\em Few sums, many products}, Studia Sci. Math. Hungar. 40 (2003), no. 3, 301-308.
\bibitem{erdosMT} {\sc P. Erd\H{o}s}, {\em An asymptotic inequality in the theory of numbers}, Vestnik Leningrad. Univ. 15 (1960), 41-49.
\bibitem{ES} {\sc P. Erd\H{o}s, E. Szemer\'edi}, {\em On sums and products of integers}, Studies in Pure Mathematics. To the memory of Paul Tur\'an (1983), Basel: Birkh\"auser Verlag, 213-218. 
\bibitem{Ford} {\sc K. Ford}, {\em Sums and products from a finite set of real numbers}, Analytic and Elementary Number Theory, Developments in Mathematics (1998), Boston, MA: Springer US, vol. 1, 59-66 
\bibitem{Ford2} {\sc K. Ford}, {\em The distribution of integers with a divisor in a given interval}, Annals of Mathematics 168 (2008), no. 2, 367-433.
\bibitem{Frei73} {\sc G. Freiman}, Foundations of a Structural Theory of Set Addition, Translations of Mathematical Monographs vol. 37 (translated from Russian), American Mathematical Society, 1973.
\bibitem{Frei59}{\sc G. Freiman}, {\em The addition of finite sets}, Izv. Vyss. Ucebn. Zaved. Matematika 6 (1959), no. 13, 202-213 (Russian) 
\bibitem{HLS} {\sc Y. Hamidoune, A. Llad\'o, O. Serra}, {\em On subsets with small product in torsion-free groups}, Combinatorica 18 (1998), no. 4, 529-540.
\bibitem{Jin07} {\sc R. Jin}, {\em Freiman's inverse problem with small doubling property}, Advances in Mathematics 216 (2007), 711-752.
\bibitem{KS1} {\sc S. Konyagin, I. Shkredov}, {\em New results on sums and products in $\R$}, Proceedings of the Steklov Institute of Mathematics 294 (2016), no. 1, 78-88.
\bibitem{KS2} {\sc S. Konyagin, I. Shkredov}, {\em On sum sets of sets having small product set}, Proceedings of the Steklov Institute of Mathematics 290 (2015), no. 1, 288-299.
\bibitem{MRSS} {\sc B. Murphy, M. Rudnev, I. Shkredov, Y. Shteinikov}, {\em On the few products, many sums problem}, J. Th\'eor.
Nombres Bordeaux 31 (2019), no. 3, 573-602.
\bibitem{Nath} {\sc M. Nathanson}, {\em On sums and products of integers}, Proc. of the Amer. Math. Soc. 125 (1997), no. 1, 9-16. 
\bibitem{RudSS}{\sc M. Rudnev, I. Shkredov, S. Stephens}, {\em On the energy variant of the sum-product conjecture}, Revista Matem\'atica Iberoamericana 36 (2019), no. 1, 207-232.
\bibitem{RudStev} {\sc M. Rudnev, S. Stephens}, {\em An update on the sum-product problem}, Mathematical Proceedings of the Cambridge Philosophical Society 173 (2022), no. 2, 411-430.
\bibitem{Soly} {\sc J. Solymosi}, {\em Bounding multiplicative energy by the sumset}, Advances in Mathematics 222 (2009), no. 2, 491-494.
\bibitem{Shak} {\sc G. Shakan}, {\em On higher energy decompositions and the sum-product phenomenon}, Mathematical Proceedings of the Cambridge Philosophical Society 167 (2018), no. 3, 599-617.
\bibitem{TaoVu} {\sc T. Tao, V. Vu}, Additive Combinatorics, Cambridge University Press, 2006.

\end{thebibliography}
\end{document}